\documentclass[reqno,11pt]{amsart}
\usepackage[foot]{amsaddr}
\usepackage{bbold}
\usepackage{charter}
\usepackage[margin=1in]{geometry}
\usepackage[colorlinks=true,linktoc=all,linkcolor=blue,citecolor=blue]{hyperref}

\theoremstyle{plain}
\newtheorem{theorem}{Theorem}[section]
\newtheorem*{theorem*}{Theorem}
\newtheorem{proposition}[theorem]{Proposition}
\newtheorem{lemma}[theorem]{Lemma}
\newtheorem{corollary}[theorem]{Corollary}
\theoremstyle{definition}

\newtheorem{definition}[theorem]{Definition}
\newtheorem*{definition*}{Definition}
\newtheorem{remark}[theorem]{Remark}

\numberwithin{equation}{section}
\allowdisplaybreaks

\newcommand{\C}{\mathbb{C}}
\newcommand{\D}{\mathbb{D}}
\newcommand{\N}{\mathbb{N}}
\newcommand{\Q}{\mathbb{Q}}

\newcommand{\Z}{\mathbb{Z}}
\newcommand{\Lip}{\mathcal{L}}
\newcommand{\LipW}{\Lip_{\textbf{w}}}
\newcommand{\LipWo}{\Lip_{\textbf{w},0}}
\newcommand{\lipwnorm}[1]{\left\|#1\right\|_{\textbf{w}}}
\newcommand{\Bloch}{\mathcal{B}}

\renewcommand{\a}{\alpha}
\renewcommand{\b}{\beta}
\renewcommand{\c}{\chi}
\renewcommand{\d}{\delta}
\newcommand{\e}{\varepsilon}
\newcommand{\f}{\varphi}

\renewcommand{\l}{\lambda}
\newcommand{\m}{\mu}

\renewcommand{\o}{\omega}

\newcommand{\s}{\sigma}

\renewcommand{\O}{\Omega}

\newcommand{\ben}{\begin{eqnarray}}
	\newcommand{\eeqn}{\end{eqnarray}}

\title[Multiplication Operators on the Weighted Lipschitz Space]{Multiplication Operators on the Weighted Lipschitz Space of a Tree}
\dedicatory{In memory of Walter Rudin}

\author{Robert F.~Allen\textsuperscript{1}, Flavia Colonna\textsuperscript{2}, and Glenn R.~Easley\textsuperscript{3}}
\address{\textsuperscript{1}Department of Mathematics, University of Wisconsin-La Crosse}
\address{\textsuperscript{2}Department of Mathematical Sciences, George Mason University}
\address{\textsuperscript{3}System Planning Corporation}
\email{allen.rob3@uwlax.edu, fcolonna@gmu.edu, geasley@sysplan.com}


\begin{document}

\begin{abstract}
We study the multiplication operators on the weighted Lipschitz space $\LipW$ consisting of the complex-valued functions $f$ on the set of vertices of an infinite tree $T$  rooted at $o$ such that $\sup_{v\ne o}|v||f(v)-f(v^-)|<\infty$, where $|v|$ denotes the distance between $o$ and $v$  and $v^-$ is the neighbor of $v$ closest to $o$. For the multiplication operator, we characterize boundedness, compactness, provide estimates on the operator norm and the essential norm, and determine the spectrum.  We prove that there are no isometric multiplication operators or isometric zero divisors on $\LipW$.
\end{abstract}

\maketitle

\section{Introduction}
Let $X$ be a Banach space of complex-valued functions on a set $\Omega$.  For a complex-valued function $\psi$ with domain $\Omega$, we define the \emph{multiplication operator with symbol $\psi$} on $X$ to be $M_\psi f = \psi f$ for $f \in X.$  The study of such operators with symbol attempts to tie the properties of the operator with the function theoretic properties of the symbol. The operator properties typically considered are boundedness, compactness, and being an isometry. Other aspects of interest are the determination of estimates on the operator norm as well as on the essential norm, and the identification of the spectrum and the essential spectrum.

A setting that has been widely considered in the literature is when $\O$ is the open unit disk $\D$ and $X$ is a Banach space of analytic functions on $\D$. Examples of such Banach spaces are the Hardy space $H^p$, the Bergman space $A^p$, and the Bloch space $\mathcal{B}$ (see \cite{Zhu:07} for more information on the operator theory on these spaces.)

In recent years, researchers have been developing versions of these spaces where the set $\Omega$ is a discrete space such as a tree or a discrete group. Historically, the function theory on trees has been largely devoted to studying the eigenfunctions of the Laplace operator (and in particular, the harmonic functions), defined as the averaging operator (with respect to a nearest-neighbor transition probability) at the neighbors of a vertex, minus the identity operator.

The study of the harmonic functions on discrete structures can be traced back many years in the literature. It was the harmonic analysis on trees developed by Cartier in \cite{Cartier:72} that made evident the analogy between trees endowed with the edge-counting metric and the open unit disk in the complex plane under the Poincar\'e  metric.

The Hardy spaces $H^p$ on trees have been studied by Kor\'{a}nyi, Picardello, and Taibleson in \cite{KoranyiPicardelloTaibleson:88}, and the theory of the $H^p$ spaces was further developed in \cite{DiBiasePicardello:95} by Di Biase and Picardello in the special case when the tree is homogeneous (that is, the vertices have the same number of neighbors).

Operators on discrete structures other than the Laplacian have been studied in a number of papers (e.g., see the works of Pavone \cite{Pavone:92}, \cite{Pavone:92-II}, \cite{Pavone:93}, Roe \cite{Roe:05}, and Rabinovich and Roch \cite{RabinovichRoch:07}, \cite{RabinovichRoch:08}, and \cite{RabinovichRoch:10}). Examples include  the composition operators on $L^p$ spaces associated with homogeneous trees, the Toeplitz operators on discrete groups, and the band-dominated operators defined on $\ell^p(X)$, where $X$ is a discrete metric space. The band-dominated operators on $\ell^p(X)$ are compositions of shift operators on $X$ with multiplication operators with symbols in $\ell^\infty(X)$ and have a natural connection to Schr\"odinger operators when $X$ is a graph.

In \cite{CohenColonna:92}, Cohen and the second author defined the Bloch space on an isotropic homogeneous tree $T$ by considering the harmonic functions which are Lipschitz when regarded as function between metric spaces, where the distance on $T$ counts the edges between pairs of vertices and $\C$ is endowed with the Euclidean distance. However, in \cite{CohenColonna:94},
where embeddings of homogeneous trees of even degree in the hyperbolic disk were constructed so that the edges are geodesic arcs of the same hyperbolic length, it was shown that the harmonicity condition on a tree from a nearest-neighbor perspective is not related to the classical harmonicity (and hence analyticity) condition on the disk derived through  interpolation.

This suggests that for the purpose of the study of certain operators with symbol such as the multiplication, or more generally, the weighted composition operators, these spaces are not natural analogues of their continuous counterparts. In particular, the study of multiplication operators on such spaces is of no interest, since in order for a multiplication operator to preserve harmonicity on a tree, its symbol must be a constant function. So, for the study of the theory of such operators with symbol, the spaces of functions on trees need to be less restrictive.

In \cite{Colonna:89} (see also \cite{Zhu:07}) it was shown that the analytic functions $f:\D\to\C$ such that \begin{equation*}\b_f=\displaystyle \sup_{z \in \D} (1-|z|^2)|f'(z)| < \infty\end{equation*} are precisely the Lipschitz functions with respect to the Poincar\'e distance $\rho$ on $\D$ and the Euclidean distance on $\C$ and $\b_f$ is the Lipschitz constant of $f$, namely
\begin{equation*}\b_f=\sup_{z\ne w}\frac{|f(z)-f(w)|}{\rho(z,w)}.\end{equation*} The collection of such functions is called the \emph{Bloch space}.

In \cite{ColonnaEasley:10}, the last two authors defined the Lipschitz space
$\mathcal{L}$ on an infinite tree $T$ rooted at vertex $o$ to be the collection of all complex-valued functions on the vertices of the tree that are Lipschitz with respect to the edge-counting metric on $T$ and the Euclidean metric on $\C$.  They showed these are precisely the functions $f$ for which \begin{equation*}\sup_{v \in T^*} |f(v)-f(v^-)| < \infty,\end{equation*} where $T^* = T\setminus\{o\}$.  It was shown that $\mathcal{L}$ is a functional Banach space under the norm
\begin{equation*}\|f\|_\Lip=|f(o)|+\sup_{v\in T^*}|f(v)-f(v^-)|,\end{equation*} and the multiplication operator was studied in detail on $\mathcal{L}$ as well as on a closed separable subspace called the \emph{little Lipschitz space}.  The space $\Lip$ can be viewed as a discrete analogue of the space $\mathcal{B}$.

In this work, we carry out the study of the multiplication operators on the space $\LipW$ of the complex-valued functions $f$ on an infinite tree $T$ rooted at $o$ satisfying the condition \begin{equation*}\sup_{v\in T^*}|v||f(v)-f(v^-)|<\infty,\end{equation*}
where $|v|$ is the number of edges in the unique path from $o$ to $v$ and $v^-$ is the neighbor of $v$ closest to $o$. The interest in studying this space is due to the fact that the bounded functions in $\LipW$ are the symbols of the bounded multiplication operators on $\Lip$ \cite{ColonnaEasley:10}. The space $\LipW$ (where the subscript \textbf{w} stands for {\textit{weight}}) can be regarded as a discrete analogue of the weighted Bloch space $\Bloch_\ell$ defined as the set of analytic functions $f$ on $\D$ such that
\begin{equation*}\sup_{z\in\D}(1-|z|^2)\log\frac{2}{1-|z|^2}|f'(z)|<\infty,\end{equation*} since the logarithmic weight is closely related to the Poincar\'{e} distance \begin{equation*}\rho(0,z)=\frac12\log\frac{1+|z|}{1-|z|}.\end{equation*}
The multiplication operators and cyclic vectors on the weighted Bloch space were studied by Ye in \cite{Ye:06}. In this work, we prove the discrete counterparts of several results in \cite{Ye:06} and expand the scope of the analysis of such operators.

\subsection{Organization of the Paper}
After giving some preliminary definitions and notation on trees, in Section \ref{Section:WLS}, we show that $\LipW$ is a Banach space under the norm \begin{equation*}\lipwnorm{f} = |f(o)| + \displaystyle\sup_{v \in T^*} |v||f(v)-f(v^-)|\end{equation*} and define a particular closed subspace $\LipWo$ we call the \emph{little weighted Lipschitz space}. We also give some useful properties that will be needed in the following sections.  In Section \ref{Section:Cyclic}, we define the notion of a cyclic vector for $\LipWo$ and determine a class of cyclic vectors.

In Section \ref{Section:Bound}, we characterize the bounded multiplication operators $M_\psi$ on $\LipW$ and $\LipWo$ in terms of the symbol $\psi$ and establish estimates on the operator norm in Section \ref{Section:Norm}. In Section \ref{Section:Spectrum}, we determine the spectrum, the point spectrum and the approximate spectrum of $M_\psi$. We also show that $M_\psi$ is bounded below if and only if the modulus of $\psi$ is bounded away from 0.

In Section \ref{Section:Compact}, we characterize the compact multiplication operators on $\LipW$ and $\LipWo$ in terms of a little-oh condition corresponding to the big-oh condition for boundedness.  In Section \ref{ess_norm}, we determine estimates on the essential norm of $M_\psi$.

In Section~\ref{isometries}, we characterize the isometric multiplication operators on $\LipW$ and $\LipWo$ and show that there are no isometric zero divisors on these spaces.

\subsection{Preliminary Definitions and Notation}
By a \emph{tree} $T$ we mean a locally finite, connected, and simply-connected graph, which, as a set, we identify with the collection of its vertices. By a \emph{function on a tree} we mean a complex-valued function on the set of its vertices. Two vertices
$v$ and $w$ are called \emph{neighbors} if there is an edge $[v,w]$
connecting them, and we use the notation $v\sim w$. A vertex is called \emph{terminal} if it has a unique neighbor. A \emph{path} is a finite or infinite sequence of vertices $[v_0,v_1,\dots]$ such that $v_k\sim v_{k+1}$ and  $v_{k-1}\ne v_{k+1}$, for all $k$. Given a tree $T$ rooted at $o$ and a vertex $v\in T$, a vertex $w$ is called a \emph{descendant} of $v$  if $v$ lies in the unique path from $o$ to $w$. The vertex $v$ is then called an \emph{ancestor} of $w$.  The vertex $v$ is called a \emph{child} of $v^-$.

For $v\in T$, the set $S_v$ consisting of $v$ and all its descendants is called the \emph{sector} determined by $v$. Define the \emph{length} of a finite path
$[v=v_0,v_1,\dots,w=v_n]$ (with $v_k\sim v_{k+1}$ for $k=0,\dots, n-1$) to be the number $n$ of edges connecting $v$ to $w$. The \emph{distance}, $d(v,w)$, between vertices $v$ and $w$ is the length of the unique path connecting $v$ to $w$. Fixing $o$ as the root of the tree, we define the \emph{length} of a vertex $v$, by $|v|=d(o,v)$.

In this paper, we shall assume the tree $T$ to be without terminal vertices (and hence infinite), and rooted at a vertex $o$ and shall denote by $L^\infty$ the space of the bounded functions $f$ on the tree equipped with the supremum norm \begin{equation*}\|f\|_\infty=\sup\limits_{v\in T}|f(v)|.\end{equation*}

\section{The Weighted Lipschitz Space}\label{Section:WLS}
Let $T$ be a tree and let $\LipW$ denote the set of functions $f$ on $T$ such that
$\sup\limits_{v\in T^*}|v|Df(v)<\infty,$ where $Df(v)=|f(v)-f(v^-)|$ for $v\in T^*$.
For $f\in \LipW$, define
\rm{\begin{equation*}\lipwnorm{f}=|f(o)|+\sup_{v\in T^*}|v|Df(v).\end{equation*}}

\begin{proposition}\label{modulus_est} If \rm{$f\in \LipW$} and $v\in T^*$, then
\rm{\ben |f(v)|\le (1+\log |v|)\lipwnorm{f}.\label{modulusest}\eeqn}
\end{proposition}

For the proof we need the following result.

\begin{lemma}\label{easy} For $x>1$, we have
\begin{equation*}\frac1{x}\le \log\left(\frac{x}{x-1}\right)\le \frac1{x-1}.\end{equation*}
\end{lemma}

\begin{proof} The upper estimate is an immediate consequence of the inequality
\begin{equation*}\log\left(1+\frac{1}{x-1}\right)\le \frac1{x-1}.\end{equation*}
The lower estimate follows from the fact that the function $\f(x)=x\log\left(\displaystyle\frac{x}{x-1}\right)$ is decreasing and $\displaystyle\lim\limits_{x\to\infty}\f(x)=1$.\end{proof}

%
\begin{proof}[Proof of Proposition \ref{modulus_est}] Let us first assume $f(o)=0$ and argue by induction on $|v|$. For $|v|=1$, we have \begin{equation*}|f(v)|=|v|Df(v)\le \lipwnorm{f}=(1+\log|v|)\lipwnorm{f}.\end{equation*}
Let $n\in\N$ and assume $|f(w)|\le (1+\log|w|)\lipwnorm{f}$ whenever $w$ is a vertex such that $1\le |w|<n$. Let $v$ be a vertex of length $n$. Then, by Lemma~\ref{easy} we get
\ben |f(v)|&\le &|f(v^-)|+|f(v)-f(v^-)|\le \left(1+\log(|v|-1)+\frac1{|v|}\right)\lipwnorm{f}\nonumber\\
&\le &(1+\log|v|)\lipwnorm{f}.\nonumber\eeqn

On the other hand, if $f(o)\ne 0$, let $g(v)=f(v)-f(o)$ for $v\in T$. By the previous case, we have $|g(v)|\le (1+\log|v|)\lipwnorm{g}$ for $v\in T^*$. Since $\lipwnorm{f}=|f(o)|+\lipwnorm{g}$, we deduce that \begin{equation*}|f(v)|\le |f(o)|+|g(v)|\le |f(o)|+(1+\log|v|)\lipwnorm{g}\le (1+\log|v|)\lipwnorm{f},\end{equation*} completing the proof. \end{proof}

\begin{theorem}\label{Banach} $\rm{\LipW}$ is a Banach space under the norm $\lipwnorm{\cdot}$.
\end{theorem}

\begin{proof} It is immediate to see that $\LipW$ is a vector space and that $f\mapsto \lipwnorm{f}$ is a semi-norm. It is also evident that the norm of the function identically 0 is 0. Conversely, assume $\lipwnorm{f}=0$. Then $Df$ is identically 0. Thus, $f$ is a constant and since $f(o)=0$, $f$ is identically 0.

To prove that $\LipW$ is a Banach space, let $\{f_n\}$ be Cauchy in $\LipW$. For $n,m\in \N$, since $|f_n(o)-f_m(o)|\le \lipwnorm{f_n-f_m}$, and by Proposition~\ref{modulus_est}, for $v\in T^*$,
\ben |f_n(v)-f_m(v)|\le (1+\log |v|)\lipwnorm{f_n-f_m},\nonumber\eeqn
the sequence $\{f_n(v)\}$ is Cauchy for each $v\in T$. Hence it converges pointwise to some function $f$. We now show that $f\in\LipW$.

Let $v\in T^*$ and fix $n\in\N$. Then
\ben\label{Eq1}|v|Df(v) \le |v||f(v)-f_n(v)|+|v|Df_n(v)+|v||f_n(v^-)-f(v^-)|.\eeqn
Since for each $v\in T^*$, $|v|Df_n(v)\le \lipwnorm{f_n}$ and $\{f_n\}$ is Cauchy in $\LipW$, and hence bounded, $\{|v|Df_n(v)\}$ is uniformly bounded by some constant $C$, and so (\ref{Eq1}) yields \begin{equation*}|v|Df(v)\le \liminf_{n\to\infty}|v|Df_n(v)\le C.\end{equation*} Hence $f\in\LipW$.

To conclude the proof of the completeness, we need to show that $f_n$ converges to $f$ in norm as $n\to\infty$.  Since $f_n(o)\to f(o)$, it suffices to show that \begin{equation*}\sup\limits_{v\in T^*}|v|D(f_n-f)(v)\to 0\end{equation*} as $n\to\infty$. Arguing by contradiction, suppose there exist $\e>0$ and a subsequence $\{f_{n_j}\}_{j\in\N}$
 such that $\displaystyle\sup_{v\in T^*}|v|D(f_{n_j}-f)(v)>\e$ for all $j\in\N$. Then for each $j\in\N$, we may pick two neighboring vertices $v_{n_j}$ and $w_{n_j}$, with $v_{n_j}$ child of $w_{n_j}$, such that \begin{equation*}|v_{n_j}||f_{n_j}(v_{n_j})-f(v_{n_j})-(f_{n_j}(w_{n_j})-f(w_{n_j}))|\ge \e.\end{equation*}  Since $\{f_{n_j}\}$ is Cauchy in $\LipW$, there exists a positive integer $j_0$ such that for each $j,h\ge j_0$, and $v\in T^*$, we have
\begin{equation*}|v||f_{n_j}(v)-f_{n_h}(v)-(f_{n_j}(v^-)-f_{n_h}(v^-))|\le \lipwnorm{f_{n_j}-f_{n_h}}<\frac{\e}{2}.\end{equation*} In particular, for all $h\ge j_0$, we have \ben |v_{n_{j_0}}||f_{n_{j_0}}(v_{n_{j_0}})-f_{n_h}(v_{n_{j_0}})-(f_{n_{j_0}}(w_{n_{j_0}})-f_{n_h}(w_{n_{j_0}}))|<\frac{\e}{2}.\label{firstest}\eeqn On the other hand, by the pointwise convergence of $f_{n_h}$ to $f$, for all integers $h$ sufficiently large
\ben |f_{n_h}(v_{n_{j_0}})-f(v_{n_{j_0}})-(f_{n_{h}}(w_{n_{j_0}})-f(w_{n_{j_0}}))|<\frac{\e}{2|v_{n_{j_0}}|}.\label{sndest}\eeqn
Thus, by the triangle inequality, from (\ref{firstest}) and (\ref{sndest}) we deduce that
\begin{equation*}|v_{n_{j_0}}||f_{n_{j_0}}(v_{n_{j_0}})-f(v_{n_{j_0}})-(f_{n_{j_0}}(w_{n_{j_0}})-f(w_{n_{j_0}}))|<\e,\end{equation*} contradicting the choice of $v_{n_{j_0}}$ and $w_{n_{j_0}}$. Therefore $\LipW$ is a Banach space.\end{proof}

A Banach space $X$ of complex-valued functions on a set $\Omega$ is said to be a \emph{functional Banach space} if for each $\o\in \Omega$, the point evaluation functional \begin{equation*}e_\o(f) = f(\o),\quad f \in X,\end{equation*} is bounded; that is, there exists a constant $C>0$ such that $|f(\o)|\le C\|f\|$, for each $f\in X$.

\begin{lemma}[Lemma 11 of \cite{DurenRombergShields:69}]\label{drs} Let $X$ be a functional Banach space on the set $\Omega$ and let $\psi$ be a complex-valued function on $\Omega$ such that $M_\psi$ maps $X$ into itself. Then $M_\psi$ is bounded on $X$ and $|\psi(\o)|\le \|M_\psi\|$ for all $\o\in\O$. In particular, $\psi$ is bounded.\end{lemma}

\begin{corollary}\label{funct_Banach} The set \rm{$\LipW$} is a functional Banach space. If $M_\psi$ is a multiplication operator on \rm{$\LipW$}, then $M_\psi$ is bounded, its symbol $\psi$ is bounded and $\|\psi\|_\infty\le \|M_\psi\|$.
\end{corollary}

\begin{proof} Let $f\in\LipW$. Then $|f(o)|\le \lipwnorm{f}$ and fixing $v\in T^*$, inequality (\ref{modulusest}) shows that the point evaluation functional $e_v(f)=f(v)$ is bounded. Thus, $\LipW$ is a functional Banach space. The second statement is an immediate consequence of Lemma~\ref{drs}.\end{proof}

Define the {\emph{little weighted Lipschitz space} to be the subspace $\LipWo$ of $\LipW$ consisting of the functions $f$ such that \begin{equation*}\lim_{|v|\to\infty}|v|Df(v)=0.\end{equation*}

\begin{proposition}\label{lm0} If \rm{$f\in\LipWo$}, then $\lim\limits_{|v|\to\infty}\frac{f(v)}{\log|v|}=0.$
\end{proposition}

\begin{proof} If $f$ is constant then the result holds trivially. Assume $f$ is nonconstant, so that $\b_f=\sup\limits_{v\in T^*}|v|Df(v)>0$, and fix $\e\in (0,\b_f)$. Then, there exists $N\in\N$ such that $|v|Df(v)<\e$, for all $v\in T$, with $|v|\ge N$. For $|w|=N$ and $v$ a descendant of $w$, let $u_0=w,u_1,\dots,u_{|v|-N}=v$ be the vertices in the path from $w$ to $v$, where ${u_j}^-=u_{j-1}$, $j=1,\dots,|v|-N$. By the triangle inequality and Proposition~\ref{modulus_est}, we have
\ben |f(v)|&\le &|f(w)|+\sum_{j=1}^{|v|-|w|}|f(u_j)-f(u_{j-1})|\nonumber\\
&\le &(1+\log N)\lipwnorm{f}+\e\sum_{k=N+1}^{|v|}\frac{1}{k}\nonumber\\
&\le &\left(1+\sum_{k=1}^{N-1}\frac1{k}\right)\lipwnorm{f}-\e\sum_{k=2}^N\frac1{k}+\e\sum_{k=2}^{|v|}\frac1{k}\nonumber\\
&<&  2\lipwnorm{f}+(\lipwnorm{f}-\e)\sum_{k=2}^N\frac1{k}+\e\log|v|.\nonumber\eeqn
Therefore, for all vertices $v$ of length greater than $N$ we obtain
\begin{equation*}\frac{|f(v)|}{\log|v|}<\frac{2\lipwnorm{f}+(\lipwnorm{f}-\e)\displaystyle\sum_{k=2}^N\frac1{k}}{\log|v|}+\e.\end{equation*}
Hence $\lim\limits_{|v|\to\infty}\displaystyle\frac{|f(v)|}{\log|v|}\le \e$. Letting $\e\to 0$, we obtain the result.\end{proof}

The following result will be used in Section~8 to derive estimates on the essential norm of the multiplication operators on $\LipW$.

\begin{proposition}\label{weakconv} Let $\{f_n\}$ be a sequence of functions in \rm{$\LipWo$} converging to $0$ pointwise in $T$ and such that \rm{$\lipwnorm{f_n}$} is bounded. Then $f_n\to 0$ weakly in \rm{$\LipWo$}.
\end{proposition}

\begin{proof} First suppose $f_n(o)=0$ for all $n\in\N$, thus
$\lipwnorm{f_n}=\sup\limits_{v\in T^*}|v|Df_n(v)$. Then, letting $\m(v)=|v|$ for $v\in T$, the sequence $\{\m Df_n\}$ converges to 0 pointwise. Observe that the subspace of $\LipWo$ whose elements send $o$ to 0 is isomorphic to the space $c_0$, consisting of the sequences indexed by $T$ which vanish at infinity, under the supremum norm via the correspondence $f\mapsto \m Df$. The space $c_0$ has dual isomorphic to the space $\ell^1$ of absolutely summable sequences (e.g. \cite{Conway:07}) via the correspondence $g\in \ell^1\mapsto \widetilde{g}\in c_0^*$, where for $f\in c_0$, \begin{equation*}\widetilde{g}(f)=\sum_{v\in T}f(v)g(v).\end{equation*}  Thus, under the identification of $\LipWo$ with $c_0$, if $f_n\in c_0$ converges pointwise to 0 and is bounded in $c_0$, then for any $g\in \ell^1$, we have
\ben |\widetilde{g}(f_n)|=\left|\sum_{v\in T}f_n(v)g(v)\right|\le \sum_{v\in T}|f_n(v)||g(v)|.\label{estgfn}\eeqn
Let $c=\sup\limits_{n\in \N,v\in T}|f_n(v)|$. Fixing any positive integer $N$, we may split the sum on the right-hand side of (\ref{estgfn}) into the two sums \begin{equation*}S_1(n,N)=\sum_{|v|\le N}|f_n(v)||g(v)|\ \hbox{ and }\ S_2(n,N)=\sum_{|v|> N}|f_n(v)||g(v)|.\end{equation*} Since $f_n\to 0$ uniformly on the set $\{v\in T: |v|\le N\}$, we see that
\begin{equation*}S_1(n,N)\le \max_{|v|\le N}|f_n(v)|\|g\|_1\to 0,\hbox{ as }n\to\infty.\end{equation*}
On the other hand, since $g\in \ell^1$, the tail end of the series $\displaystyle\sum_{v\in T}|g(v)|$ approaches 0.
Therefore, since
\begin{equation*}\lim_{n\to\infty}|\widetilde{g}(f_n)|\le \lim_{n\to\infty}S_1(n,N)+\sup_{n\in\N}S_2(n,N)\le c\sum_{|v|>N}|g(v)|,\end{equation*} letting $N\to \infty$, we deduce that $\displaystyle\lim_{n\to\infty}\widetilde{g}(f_n)=0$.

Hence, if $f_n(o)=0$, then $f_n$ converges to 0 weakly. In the general case, define $F_n=f_n-f_n(o)$. By the previous part, $F_n\to 0$ weakly. Since $f_n(o)\to 0$, we conclude that $f_n\to 0$ weakly as well.\end{proof}

Denote by $\chi_A$ the characteristic function of the set $A$ and use the simpler notation $\chi_v$ for the function $\chi_{\{v\}}$.

\begin{proposition}\label{dense} The set \begin{equation*}\mathcal{P}=\left\{\sum\limits_{k=1}^N a_kp_{v_k}: N\in\N, a_k\in\C, v_k\in T, 1\le k\le N\right\},\end{equation*} is dense in \rm{$\LipWo$}, where $p_v=\c_{S_v}$ for $v\in T$.
\end{proposition}

\begin{proof} Fix $v\in T$ and observe that $Dp_v=\c_v$, so that for $w\in T^*$, we have
\begin{equation*}|w|Dp_v(w)=\begin{cases} 0&\quad\hbox{ if }w\ne v,\\
|v|& \quad\hbox{ if }w= v.\end{cases}\end{equation*}
Thus, as $|w|\to\infty$, $|w|Dp_v(w)\to 0$, proving that $p_v\in\LipWo$.

Let $f\in\LipWo$ and for $n\in\N$, define
\begin{equation*}f_n(v)=\begin{cases} f(v)& \quad\hbox{ if }|v|\le n,\\
f(v_n)&\quad\hbox{ if }|v|>n,\end{cases}\end{equation*}
where $v_n$ is the ancestor of $v$ of length $n$. Observe that for $v\in T^*$,
\ben \chi_v=p_v-\sum_{w\in v^+}p_w,\label{chi_v}\eeqn
where $v^+=\{w\in T:w^-=v\}$.
Therefore, for $n\in\N$, we have
\ben f_n&=&\sum_{|v|<n}f(v)\c_v +\sum_{|v|=n}f(v)p_v\nonumber\\
&=&\sum_{|v|<n}f(v)\left(p_v-\sum_{w\in v^+}p_w\right) +\sum_{|v|=n}f(v)p_v\nonumber\\
&=&\sum_{|v|\le n}f(v)p_v-\sum_{|v|<n}f(v)\sum_{w\in v^+}p_w\nonumber\eeqn
Thus, $f_n$ is a finite linear combination of the functions $p_v$ and
\begin{equation*}\lipwnorm{f_n-f}=\sup_{|v|>n}|v|Df(v)\to 0\end{equation*}
as $n\to\infty$, proving the result.\end{proof}

\begin{remark} Since $\Q[i]=\{z\in\C: {\rm{Re}}\, z, {\rm{Im}}\, z\in\Q\}$ is dense in $\C$, and $T$ is countable, the subset of $\mathcal{P}$ consisting of the finite linear combinations of the functions $p_v$ with coefficients in $\Z[i]$ is countable and dense in \rm{$\LipWo$}. Therefore, \rm{$\LipWo$} is a closed separable subspace of \rm{$\LipW$}.
\end{remark}

\section{Cyclic Vectors in the Weighted Little Lipschitz Space}\label{Section:Cyclic}
\begin{definition}\label{cyclic} Let $X$ be a Banach space of functions on $T$ such that $\mathcal{P}$ is dense in $X$. A function $f$ in $X$ is called a {\textit{cyclic vector}} if $X$ is the closure $[f]$ of the functions of the form $p_vf$.
\end{definition}

	If $f\in\LipWo$ vanishes at some vertex $u$, then $f$ cannot be a cyclic vector since the function $\c_u$ cannot be the limit in $\LipWo$ of multiples of $f$. For the converse, we have the following result.

\begin{theorem}\label{cyclic_in_lm0} Let \rm{$f\in\LipWo$} be such that $|f(v)|\ge \d>0$ for all $v\in T$. Then $f$ is a cyclic vector in \rm{$\LipWo$}.
\end{theorem}

\begin{proof} First observe that to prove the result, it suffices to show that the constant function 1 is a limit in $\LipWo$ of functions of the form $p_vf$. Indeed, observe that if $v,u\in T$, then
\begin{equation*}p_vp_u=\begin{cases}0& \quad\hbox{ if }S_v\cap S_u=\emptyset,\\
p_v& \quad\hbox{ if }u \hbox{ is an ancestor of }v,\\
p_u& \quad\hbox{ if }v \hbox{ is an ancestor of }u \hbox{ or }v=u.\end{cases}\end{equation*} Thus, $1\in [f]$ implies that $p_v\in [f]$ for all $v\in T$. By Proposition~\ref{dense}, it follows that $f$ is a cyclic vector in $\LipWo$.

For $n\in\N$, define $f_n$ as in the proof of Proposition~\ref{dense}. Then
\ben\lipwnorm{\frac{f}{f_n}-1}=\sup_{|v|>n}{|v|}\frac{Df(v)}{|f(v_n)|}\le \frac1{\d}\sup_{|v|>n}|v|Df(v)\to 0\nonumber\eeqn
as $n\to\infty$. On the other hand,
letting \begin{equation*}g_{n}=\sum_{|v|\le n^2}\frac1{f_n(v)}\c_v,\end{equation*} we see that
\begin{equation*}\lipwnorm{g_nf-\frac{f}{f_n}}=\sup_{|v|>n^2}\frac{|v|Df(v)}{f(v_n)}\le \frac1{\d}\sup_{|v|>n^2}|v|Df(v)\to 0\end{equation*}
as $n\to\infty$. Thus, $\lipwnorm{g_nf-1}\to 0$ as $n\to\infty$. For $v\in T^*$, recalling (\ref{chi_v}), we see that the functions $g_nf$ belong to $[f]$. Therefore, $f$ is a cyclic vector.\end{proof}

It is still an open question as to whether there exist cyclic vectors that are not bounded away from 0.

\section{Boundedness of $M_\psi$}\label{Section:Bound}
In this section, we characterize the bounded multiplication operators acting on $\LipW$ and $\LipWo$.  This characterization provides a big-oh criterion for boundedness, which corresponds to a little-oh criterion for compactness developed in Section \ref{Section:Compact}.

\begin{theorem}\label{boundedness} For a function $\psi$ on $T$ the following statements are equivalent.
\begin{enumerate}
\item[\normalfont{(a)}] $M_\psi$ is bounded on \rm{$\LipW$}.

\item[\normalfont{(b)}] $M_\psi$ is bounded on \rm{$\LipWo$}.

\item[\normalfont{(c)}] $\psi\in L^\infty$ and $\displaystyle\sup\limits_{v\in T^*}|v|\log|v|D\psi(v)<\infty.$
\end{enumerate}
\end{theorem}

\begin{proof} We first prove (a)$\Longrightarrow$(c).  Assume $M_\psi$ is bounded on $\LipW$. The boundedness of $\psi$ follows immediately from Corollary~\ref{funct_Banach}.

For $v\in T$, define \begin{equation*}f(v)=\begin{cases}\ \ 0 &\quad\hbox{ if }v=o,\\
\log|v|&\quad\hbox{ if }v\in T^*.\end{cases}\end{equation*}  Then, for $|v|=1$, we have $Df(v)=0$, while for $|v|\ge 2$, by Lemma~\ref{easy}, we obtain \begin{equation*}Df(v)=\log\left(\frac{|v|}{|v|-1}\right)\le \frac1{|v|-1}.\end{equation*}
Thus, $|v|Df(v)\le \displaystyle\frac{|v|}{|v|-1}\le 2.$ Therefore, $f\in\LipW$. Furthermore, for $v\in T^*$ we have
\ben D\psi(v)|f(v)|&\le & |\psi(v)f(v)-\psi(v^-)f(v^-)|+|\psi(v^-)f(v^-)-\psi(v^-)f(v)|\nonumber\\
&= & D(\psi f)(v)+|\psi(v^-)|Df(v).\label{prodrule2}\eeqn
Thus, by the boundedness of $M_\psi$, for $v\in T^*$, we obtain
\ben |v|D\psi(v)|f(v)|&\le &|v|D(\psi f)(v)+|\psi(v^-)||v|Df(v)\nonumber\\
&\le& \lipwnorm{M_\psi f}+\|\psi\|_\infty\lipwnorm{f}.\nonumber\eeqn
Hence \ben \sup_{v\in T^*}|v|\log|v|D\psi(v)<\infty.\label{modvlogv}\eeqn

Next, we prove (c)$\Longrightarrow$(a).  Assume $\psi$ is bounded and (\ref{modvlogv}) holds. Let $f\in\LipW$ and $v\in T^*$. Note that
\ben D(\psi f)(v)&\le & |\psi(v)f(v)-\psi(v^-)f(v)|+|\psi(v^-)f(v)-\psi(v^-)f(v^-)|\nonumber\\
&=& D\psi(v)|f(v)|+|\psi(v^-)|Df(v).\label{prodrule}\eeqn
Thus, by Proposition~\ref{modulus_est}, we have
\ben |v|D(\psi f)(v)&\le & |v|D\psi(v)|f(v)|+|\psi(v^-)||v|Df(v)\nonumber\\
&\le &|v|(1+\log|v|)D\psi(v)\lipwnorm{f} +\|\psi\|_\infty\lipwnorm{f}.\nonumber\eeqn
In particular, for $|v|\ge 3$, we have
\begin{equation*} |v|D(\psi f)(v)\le (2|v|\log|v|D\psi(v)+\|\psi\|_\infty)\lipwnorm{f},\end{equation*}
proving that $\psi f\in \LipW$. The boundedness of $M_\psi$ follows from Lemma~\ref{drs}.

Now, we prove (b)$\Longrightarrow$(c).  Assume $M_\psi$ is bounded on $\LipWo$. For $0<\a<1$, define \begin{equation*}f_\a(v)=\begin{cases} 0 &\hbox{ if } v=o,\\(\log|v|)^\a &\hbox{ if }v\in T^*.\end{cases}\end{equation*} Then $|v|Df_\a(v)\to 0$ as $|v|\to\infty$, so that $f_\a\in\LipWo$. Since for $0<\a<1$, the function $x\mapsto x-x^\a$ is increasing for $x\ge 1$, it follows that for $|v|\ge 2$, $Df_\a(v)\le \log(|v|)-\log(|v|-1)$, so by Lemma~\ref{easy}, we have
\begin{equation*}|v|Df_\a(v)\le |v|(\log|v|-\log(|v|-1))\le \frac{|v|}{|v|-1}\le 2.\end{equation*} Furthermore, for $|v|=1$, $|v|Df_\a(v)=0$. Thus, $\lipwnorm{f_\a}\le 2$ for all $\a\in (0,1)$. Moreover, by Lemma~\ref{drs}, the function $\psi$ is bounded, so by (\ref{prodrule2}), for $v\in T^*$, we have
\ben |v|D\psi(v)|f_\a(v)|&\le& |v|D(\psi f_\a)(v)+|v||\psi(v^-)|D f_\a(v)\nonumber\\
&\le& \lipwnorm{M_\psi f_\a}+\|\psi\|_\infty\lipwnorm{f_\a}.\nonumber\eeqn
This implies that \begin{equation*}|v|D\psi(v)(\log|v|)^\a\le (\|M_\psi\|+\|\psi\|_\infty)\lipwnorm{f_\a}.\end{equation*}
Letting $\a$ approach 1, by the boundedness of $\lipwnorm{f_\a}$, we obtain (\ref{modvlogv}).

Lastly, we prove (c)$\Longrightarrow$(b).  Assume (c) holds and let $f\in\LipWo$. By (\ref{prodrule}) and Proposition~\ref{lm0}, for $|v|>1$ we have
\ben |v|D(\psi f)(v)&\le &|v|D\psi(v)|f(v)|+|v||\psi(v^-)|Df(v)\nonumber\\
&\le &|v|\log|v|D\psi(v)\frac{|f(v)|}{\log|v|}+\|\psi\|_\infty|v|Df(v)\to 0\nonumber\eeqn
as $|v|\to\infty$. Therefore, $\psi f\in\LipWo$. The boundedness of $M_\psi$ on $\LipWo$ follows from Lemma~\ref{drs}.\end{proof}

\section{Norm of $M_\psi$}\label{Section:Norm}
In this section, we provide estimates on the norm of the bounded multiplication operators on $\LipW$ and $\LipWo$.

\begin{theorem}\label{norm_estimates}
Let $M_\psi$ be a bounded multiplication operator on \rm{$\LipW$} or \rm{$\LipWo$}.  Then
\begin{equation*}\max\{\lipwnorm{\psi},\|\psi\|_\infty\}\le \|M_\psi\| \leq \|\psi\|_\infty + \displaystyle\sup_{v \in T^*} |v|(1+\log|v|)D\psi(v).\end{equation*}
\end{theorem}

\begin{proof}
Let $f$ be the function identically equal to 1 on $T$. Since $Df$ is identically 0, $f \in \LipW$, and $\lipwnorm{f}=1$. Thus, $\psi=\psi f\in\LipW$ and $\lipwnorm{M_\psi f} = \lipwnorm{\psi}$. Therefore, $\lipwnorm{\psi} \leq \|M_\psi\|$. Moreover, by Lemma~\ref{drs}, $\psi\in L^\infty$ and $\|\psi\|_\infty \leq \|M_\psi\|$, proving the lower estimate.

Let $f \in \LipW$ such that $\lipwnorm{f} = 1$.  Then, using (\ref{prodrule}), Proposition~\ref{modulus_est}, and the fact that $\sup\limits_{v\in T^*}|v|Df(v)=1-|f(o)|$, we obtain
\begin{equation*}\begin{aligned}
\|M_\psi f\|
&\le |\psi(o)||f(o)| + \sup_{v \in T^*}|v||f(v)|D\psi(v) + \sup_{v \in T^*}|v||\psi(v^-)|Df(v)\\&\leq |\psi(o)||f(o)| + \sup_{v \in T^*}|v|(1 + \log|v|)D\psi(v) + \|\psi\|_\infty\sup_{v \in T^*}|v|Df(v)\\
&= |\psi(o)||f(o)| + \sup_{v \in T^*}|v|(1+\log|v|)D\psi(v)+ \|\psi\|_\infty(1-|f(o)|)\\
&\le \|\psi\|_\infty + \sup_{v \in T^*}|v|(1+\log|v|)D\psi(v),\end{aligned}\end{equation*}
proving the upper estimate.\end{proof}

\section{Spectrum of $M_\psi$}\label{Section:Spectrum}
In this section, we study the spectra of the bounded multiplication operator $M_\psi$ on $\LipW$ and $\LipWo$.  We show that the point spectrum is nonempty and, in fact, it is a dense subset of the spectrum. We also show that the spectrum and the approximate point spectrum are equal to the closure of the range of the symbol. We deduce a characterization of the bounded multiplications operators that are bounded below.

Recall that for a bounded operator $S$ on a Banach space $X$, the \emph{spectrum} of $S$ is defined as
\begin{equation*}\sigma(S) = \left\{\lambda \in \C : S - \lambda I \text{ is not invertible}\right\},\end{equation*} where $I$ is the identity operator on $X$. The \emph{point spectrum} of $S$ is defined as
\begin{equation*}\sigma_p(S) = \left\{\lambda \in \C : \mathrm{ker}(S-\lambda I) \neq \{0\}\right\}.\end{equation*}
The \emph{approximate point spectrum} of $S$ is defined as
\small \begin{equation*}\sigma_{ap}(S) = \left\{\lambda \in \C : \exists \{x_n\} \subseteq X, \text{ such that }\, \|x_n\|=1\  \forall n, \text{ and }\, \|(S-\lambda I)x_n\| \to 0\right\}.\end{equation*}
\normalsize

The following inclusions hold:
\ben \s_p(S)\subseteq \s_{ap}(S)\subseteq\s(S).\label{spectra}\eeqn

\begin{theorem}\label{spectrum} Let $M_\psi$ be a bounded multiplication operator on $\LipW$ or $\LipWo$. Then
\begin{enumerate}
\item[\normalfont{(a)}] $\sigma_p(M_\psi) = \psi(T)$.

\item[\normalfont{(a)}] $\s(M_\psi)=\overline{\psi(T)}$.
\end{enumerate}
\end{theorem}

\begin{proof} To prove (a), suppose $\lambda \in \sigma_p(M_\psi)$.  Then there exists a non-zero function $f \in \LipWo$ such that $M_{\psi-\lambda}f$ is identically zero.  Since $f$ is not identically zero, there exists $w \in T$ such that $f(w) \neq 0$.  Then $0=(M_{\psi-\lambda}f)(w) = (\psi(w)-\lambda)f(w)$, and so $\psi(w) = \lambda$, proving that $\lambda$ is in the image of $\psi$.

Conversely, suppose $\lambda$ is in the image of $\psi$.  Then there exists $w \in T$ such that $\psi(w) = \lambda$.  So we see that $M_{\psi-\lambda}\chi_{w}$ is identically zero.  Thus $\lambda \in \sigma_p(M_\psi)$.  Therefore $\sigma_p(M_\psi)=\psi(T)$.

To prove (b), observe that since the spectrum is closed, the inclusion $\overline{\psi(T)}\subseteq \s(M_\psi)$ follows at once from part (a) by passing to the closure.

Conversely, if $\l\notin \overline{\psi(T)}$, then $|\psi(v)-\l|\ge c\,$ for some positive constant $c$ and all $v\in T$. Thus, the function $g=(\psi-\l)^{-1}$ is bounded on $T$. Furthermore,
\ben \sup_{v\in T^*}|v|\log|v|Dg(v)&=&\sup_{v\in T^*}|v|\log|v|\left|\frac1{\psi(v)-\l}-\frac1{\psi(v^-)-\l}\right|\nonumber\\
&\le& \frac1{c^2}\sup_{v\in T^*}|v|\log|v|D\psi(v)<\infty.\nonumber\eeqn
By Theorem~\ref{boundedness}, we deduce that $M_g=M_{(\psi-\l)^{-1}}$ is a bounded operator on $\LipW$ or $\LipWo$, which implies that $\l\notin \s(M_\psi)$. Therefore $\s(M_\psi)=\overline{\psi(T)}$.\end{proof}

The following proposition relates the boundary of the spectrum to the approximate point spectrum.

\begin{proposition}[Proposition 6.7 of \cite{Conway:07}]\label{boundary spectrum} If $S$ is a bounded operator on a Banach space, then the boundary of $\s(S)$ is a subset of $\s_{ap}(S)$. \end{proposition}

	Using Theorem~\ref{spectrum}, the inclusions (\ref{spectra}) and Proposition~\ref{boundary spectrum}, we obtain the following result.

\begin{corollary}\label{approximate point spectrum} Let $M_\psi$ be a bounded multiplication operator on \rm{$\LipW$} or \rm{$\LipWo$}. Then $\sigma_{ap}(M_\psi)=\overline{\psi(T)}$.
\end{corollary}

A bounded operator $S$ on a Banach space $X$ is said to be \emph{bounded below} if there exists a positive constant $c$ such that $\|Sx\| \geq c\|x\|$ for all $x \in X$.  Note that a bounded operator that is bounded below is necessarily injective.

	The following result connects the approximate point spectrum and the operators that are bounded below.

\begin{proposition}[Proposition 6.4 of \cite{Conway:07}]\label{conway approximate point spectrum} If $S$ is a bounded operator on a Banach space, then $\lambda \not\in \sigma_{ap}(S)$ if and only if $S-\lambda I$ is bounded below.
\end{proposition}

	We next characterize the bounded multiplication operators on $\LipW$ or $\LipWo$ which are bounded below.

\begin{theorem}\label{bounded below} If $M_\psi$ is a bounded multiplication operator on $\LipW$ or $\LipWo$, then $M_\psi$ is bounded below if and only if $\inf\limits_{v\in T}|\psi(v)|>0$.
\end{theorem}

\begin{proof} By Proposition~\ref{conway approximate point spectrum}, if $M_\psi$ is a bounded operator on $\LipW$ or $\LipWo$, then $M_\psi$ is bounded below if and only if $0\notin \s_{ap}(M_\psi)$. By Corollary~\ref{approximate point spectrum}, this condition is equivalent to $0\notin \overline{\psi(T)}$, i.e. $\inf\limits_{v\in T}|\psi(v)|>0$.\end{proof}

\section{Compactness of $M_\psi$}\label{Section:Compact}
In this section, we characterize the compact multiplication operators on $\LipW$ and $\LipWo$.

\begin{lemma}\label{compact:chara} A bounded multiplication operator $M_\psi$ on \rm{$\LipW$} (respectively, \rm{$\LipWo$}) is compact if and only if \rm{$\lipwnorm{\psi f_n}\to 0$} as $n\to\infty$ for every bounded sequence $\{f_n\}$ in \rm{$\LipW$} (respectively, \rm{$\LipWo$}) converging to 0 pointwise.\end{lemma}

\begin{proof} We shall prove the result for the bounded operator $M_\psi$ acting on $\LipW$. The proof for the case of $\LipWo$ is analogous.

	Assume $M_\psi$ is compact on $\LipW$ and let $\{f_n\}$ be a bounded sequence in $\LipW$ converging to 0 pointwise. By rescaling the sequence, if necessary, we may assume $\lipwnorm{f_n} \le 1$ for all $n\in \N$. By the compactness of $M_\psi$, $\{f_n\}$ has a subsequence $\{f_{n_k}\}$ such that $\{\psi f_{n_k}\}$
 converges in norm to some function $f\in\LipW$. Observe that $\psi(o) f_{n_k}(o)\to f(o)$ and for $v\in T^*$, by Proposition~\ref{modulus_est} applied to the function $\psi f_{n_k}-f$, we have
\ben |\psi(v)f_{n_k}(v)-f(v)|\le (1+\log|v|)\lipwnorm{\psi f_{n_k}-f}.\nonumber\eeqn
Therefore, $\psi f_{n_k}\to f$ pointwise. Since by assumption, $f_n\to 0$ pointwise, it follows that $f$ must be identically 0, whence $\lipwnorm{\psi f_n}\to 0$. Since $0$ is the only limit point in $\LipW$ of the sequence $\{\psi f_n\}$, it follows that $\lipwnorm{\psi f_n}\to 0$ as $n\to\infty$.

	Conversely, suppose that for every bounded sequence $\{f_n\}$ in $\LipW$ converging to 0 pointwise, $\lipwnorm{\psi f_n}\to 0$ as $n\to\infty$. Let $\{g_n\}$ be a sequence in $\LipW$ with $\lipwnorm{g_n}\le 1$. Then $|g_n(o)|\le 1$ and by Proposition~\ref{modulus_est}, for each $v\in T^*$, we have $|g_n(v)|\le 1+\log|v|$. Therefore, $g_n$ is uniformly bounded on finite subsets of $T$ and so some subsequence, which for notational convenience we reindex as the original sequence, converges to some function $g$. Then, for $v\in T^*$, we have
\begin{equation*}Dg(v)\le |g(v)-g(v^-)-(g_n(v)-g_n(v^-))|+D g_n(v).\end{equation*} Fix $\e>0$ and $v\in T$, $|v|\ge 2$. Since $g_n\to g$ pointwise,
\begin{equation*}|g_n(v)-g(v)|<\e/(2|v|)\end{equation*} and \begin{equation*}|g_n(v^-)-g(v^-)|<\e/(2|v^-|)\end{equation*} for all $n$ sufficiently large. Therefore $|v|D g(v)<\e +|v|Dg_n(v)$ for $n$ sufficiently large, so $g\in\LipW$.
Therefore, the sequence $\{f_n\}$ defined by $f_n=g_n-g$ is bounded in $\LipW$ and converges to 0 pointwise; hence, by the hypothesis, $\lipwnorm{\psi f_n}\to 0$ as $n\to\infty$. We conclude that $\psi g_n\to\psi g$ in norm, proving the compactness of $M_\psi$.\end{proof}

\begin{theorem}~\label{chara_compactness} For $M_\psi$ a bounded multiplication operator on \rm{$\LipW$}, the following are equivalent statements.
\begin{enumerate}
\item[\normalfont{(a)}] $M_\psi$ is compact on \rm{$\LipW$}.

\item[\normalfont{(b)}] $M_\psi$ is compact on \rm{$\LipWo$}.

\item[\normalfont{(c)}] $\displaystyle\lim\limits_{|v|\to\infty}\psi(v)=0$ and $\displaystyle\lim\limits_{|v|\to \infty}|v|\log|v|D\psi(v)=0.$
\end{enumerate}
\end{theorem}

\begin{proof} First, we prove (a)$\Longrightarrow$(c).  Assume $M_\psi$ is compact on $\LipW$. Let $\{v_n\}$ be a sequence in $T$ such that $2< |v_n|\to\infty$. We are going to show that
\ben &\,&\lim\limits_{n\to\infty}\psi(v_n)=0 \label{f1}\\&\,&\lim\limits_{n\to \infty}|v_n|\log|v_n|D\psi(v_n)=0.\label{f2}\eeqn
Let $f_n=\displaystyle\frac1{|v_n|}\c_{v_n}$. Then $f_n\to 0$ pointwise and
\begin{equation*}\lipwnorm{f_n}=\sup_{v\in T^*}|v|Df_n(v)=\frac{|v_n|+1}{|v_n|}<\frac32,\end{equation*}
so that by Lemma~\ref{compact:chara}, we obtain $|\psi(v_n)|\le \lipwnorm{\psi f_n}\to 0$ as $n\to\infty$, proving (\ref{f1}).

To prove (\ref{f2}), for $n\in\N$ let \begin{equation*}g_n(v)=\begin{cases} 0& \quad\hbox{ if }|v|< \sqrt{|v_n|},\\
2\log|v|-\log|v_n| & \quad\hbox{ if }\sqrt{|v_n|}\le |v|<|v_n|-1,\\
\log|v_n| & \quad\hbox{ if }|v|\ge |v_n|-1.\end{cases}\end{equation*}
Then $Dg_n(v)=0$ if $|v|\le \sqrt{|v_n|}$ or $|v|>|v_n|$, and if $\sqrt{|v_n|}< |v| <|v_n|-1$ then $|v|Dg_n(v)\le 4$. Thus, $\{\lipwnorm{g_n}\}$ is bounded and $g_n\to 0$ pointwise as $n\to\infty$. By Lemma~\ref{compact:chara}, we get $|v_n|\log|v_n||\psi(v_n)|\le \lipwnorm{\psi g_n}\to 0$ as $n\to\infty$.

Next, we prove (c)$\Longrightarrow$(a).  Assume the conditions in (c) hold and set aside the case when $\psi$ is the constant 0. By Lemma~\ref{compact:chara}, to prove that $M_\psi$ is compact on $\LipW$, it suffices to show that if $\{f_n\}$ is a sequence in $\LipW$ converging to 0 pointwise and such that $s=\displaystyle\sup_{n\in\N}\lipwnorm{f_n}<\infty$, then $\lipwnorm{\psi f_n}\to 0$ as $n\to\infty$. Let $\{f_n\}$ be such a sequence and fix a positive number $\e$. Then $|f_n(o)|<\displaystyle\frac{\e}{3\lipwnorm{\psi}}$ for all $n$ sufficiently large, and there exists $M\in \N$ such that \begin{equation*}|\psi(v)|<\frac{\e}{3s}\ \hbox{ and }\ |v|\log|v|D\psi(v)<\frac{\e}{3s}\ \hbox{ for }\ |v|\ge M.\end{equation*}
If $|v|> M$, then $|v^-|\ge M$, and $|\psi(v^-)|<\displaystyle\frac{\e}{3s}$. Thus, by (\ref{prodrule}) and Proposition~\ref{modulus_est}, we obtain
\begin{equation*} |v|D(\psi f_n)(v) \le |v|D\psi(v)(1+\log|v|)\lipwnorm{f_n}+\frac{\e}3.\end{equation*}
In particular, for $|v|>M$, we get
\begin{equation*}|v|D(\psi f_n)(v) \le 2|v|\log|v|D\psi(v)s+\frac{\e}3<\e.\end{equation*}
Since $f_n\to 0$ uniformly on $\{v\in T: |v|\le M\}$ as $n\to\infty$, so does the sequence $\{|w|D(\psi f_n)(w)\}_{w\in T^*}$ (under an identification of $T$ with $\N$). Therefore, for $n$ sufficiently large and for each $v\in T^*$, $|v|D(\psi f_n)(v)<\e$. On the other hand, $f_n(o)\to 0$, and so $\lipwnorm{\psi f_n}\to 0$, as $n\to\infty$.

Note that for $n\in\N$ the functions $f_n$ and $g_n$ defined to prove (a)$\Longrightarrow$(c) are in $\LipWo$. Therefore the proof of (b)$\Longrightarrow$(c) is analogous. The converse can also be proved similarly. \end{proof}

\section{Essential Norm of $M_\psi$}\label{ess_norm}
In this section, we provide estimates on the essential norm of the bounded multiplication operators on $\LipW$.  We recall that the \emph{essential norm} $\|S\|_e$ of an operator $S$ on a Banach space $X$ is defined as \begin{equation*}\|S\|_e=\inf\{\|S-K\|: K \hbox{ compact operator on }X\}.\end{equation*}

\begin{definition} Given a bounded multiplication operator $M_\psi$ on \rm{$\LipW$} or \rm{$\LipWo$}, define
\ben A(\psi)&=&\lim_{n\to\infty}\sup_{|v|\ge\, n}|\psi(v)|\nonumber\\
B(\psi)&=&\lim_{n\to\infty}\sup_{|v|\ge\, n}|v|\log|v|D\psi(v).\nonumber\eeqn
\end{definition}

\begin{theorem}\label{lowerestimate} Let $M_\psi$ be bounded on \rm{$\LipW$} or \rm{$\LipWo$}. Then
\begin{equation*}\|M_\psi\|_e\ge \max\left\{A(\psi),B(\psi)\right\}.\end{equation*}
\end{theorem}

\begin{proof} For each $n\in\N$ and $v\in T$, define
$f_n=\displaystyle\frac1{n}\c_{\{v:\,|v|=n\}}$. Then $f_n\in\LipWo$, $\lipwnorm{f_n} =\displaystyle\frac{n+1}{n}\le 2$, and $f_n\to 0$ pointwise. Therefore, by Proposition~\ref{weakconv}, the sequence $\{f_n\}$ converges weakly to 0 in $\LipWo$. Since compact operators are completely continuous \cite{Conway:07}, it follows that $\displaystyle\lim\limits_{n\to\infty}\lipwnorm{Kf_n}=0$ for any compact operator $K$ on $\LipWo$. Therefore, if $K$ is a compact operator on $\LipWo$, then
\begin{equation*}\|M_\psi-K\|\ge \limsup_{n\to\infty}\lipwnorm{(M_\psi-K)f_n}\ge \limsup_{n\to\infty}\lipwnorm{M_\psi f_n}.\end{equation*}
Thus,
\ben \|M_\psi\|_e&\ge & \inf \{\|M_\psi-K\|:\,K \hbox{ compact on }\LipWo\} \nonumber\\
&\ge& \limsup_{n\to\infty} \lipwnorm{M_\psi f_n}\nonumber\\
&=& \limsup_{n\to\infty} \sup_{v\in T^*}|v||\psi(v)f_n(v)-\psi(v^-)f_n(v^-)|\nonumber\\
&= &\lim_{n\to\infty} \frac{n+1}{n}\sup_{|v|\ge n}|\psi(v)|=A(\psi).\nonumber\eeqn

Next we show that $\|M_\psi\|_e\ge B(\psi)$. The result is immediate if $B(\psi)=0$. So assume $\{v_n\}$ is a sequence of vertices of length greater than 1 such that $|v_n|$ is increasing unboundedly and \begin{equation*}\lim_{n\to\infty}|v_n|\log|v_n|D\psi(v_n)=B(\psi).\end{equation*}
Fix $p\in (0,1)$ and for $n\in \N$, let
\begin{equation*} h_{p,n}(v)=\begin{cases}\displaystyle\frac{(\log(|v|+1))^{p+1}}{(\log|v_n|)^p} &\hbox{ if }\quad 0\le |v|< |v_n|,\\
\log|v_n|  &\hbox{ if }\quad |v|\ge |v_n|.\end{cases}\end{equation*}
Then $h_{p,n}(o)=0$, $h_{p,n}(v_n)=h_{p,n}(v_n^-)=\log|v_n|$, and
\begin{equation*}|v|D h_{p,n}(v)=\begin{cases} \displaystyle\frac{|v|}{\log|v_n|^p}\left[(\log(|v|+1))^{p+1}-(\log|v|)^{p+1}\right] &\hbox{ if } 1\le |v|< |v_n|,\\
\ \ 0 &\hbox{ if } |v|\ge |v_n|.\end{cases}\end{equation*}
The supremum of $v\mapsto |v|Dh_n(v)$ is attained at the vertices of length $|v_n|-1$ and by a straightforward calculation it can be written as
\begin{equation*}s_{p,n}=(|v_n|-1)(\log|v_n|-\log(|v_n|-1))\left[\frac{\log(|v_n|-1)}{\log|v_n|}\frac{1-\left(\frac{\log(|v_n|-1)}{\log|v_n|}\right)^p}{1-\frac{\log(|v_n|-1)}{\log|v_n|}}+1\right].\end{equation*} Since the product of the first two factors approaches 1 as $n\to\infty$ and the function $\frac{1-u^{p}}{1-u}$ approaches $p$ as $u\to 1$, we see that $s_{p,n}\to p+1$ as $n\to\infty$.
In particular, $\lipwnorm{h_{p,n}}=s_{p,n}$ yields a bounded sequence. Letting $g_{p,n}=\displaystyle\frac{h_{p,n}}{s_{p,n}}$,  we see that $g_{p,n}\in \LipWo$, $\lipwnorm{g_{p,n}}=1$, and $g_{p,n}\to 0$ pointwise. Consequently, by Proposition~\ref{weakconv}, the sequence $\{g_{p,n}\}$ converges to $0$ weakly. This implies that $\lipwnorm{Kg_{p,n}}\to 0$ as $n\to\infty$ for any compact operator $K$ on $\LipWo$. We deduce that for any such operator $K$
\begin{equation*}\|M_\psi-K\|\ge \limsup_{n\to\infty}\lipwnorm{(M_\psi-K)g_{p,n}}\ge \limsup_{n\to\infty}\lipwnorm{\psi g_{p,n}}.\end{equation*} Therefore
\ben \|M_\psi\|_e&\ge & \inf\{\|M_\psi-K\|:\,K \hbox{ compact on }\LipWo\}\nonumber\\&\ge &\limsup_{n\to\infty}\sup_{v\in T^*}|v|D(\psi g_{p,n})(v).\label{compest1}\eeqn
Next, observe that for each $n\in \N$, $g_{p,n}(v_n)=g_{p,n}(v_n^-)=\displaystyle\frac{\log|v_n|}{s_{p,n}}$, so
\ben |v_n|D(\psi g_{p,n})(v_n)&=&|v_n||\psi(v_n)g_{p,n}(v_n)-\psi(v_n^-)g_{p,n}(v_n^-)|\nonumber\\
&=&|v_n|D\psi(v_n)\frac{\log|v_n|}{s_{p,n}}.\label{compest2}\eeqn
Therefore, from (\ref{compest1}) and (\ref{compest2}), we obtain
\ben \|M_\psi\|_e\ge \frac1{p+1}\lim_{n\to\infty}|v_n|\log|v_n|D\psi(v_n)=\frac{1}{p+1}B(\psi).\nonumber\eeqn
Finally, letting $p$ approach 0, we deduce $\|M_\psi\|_e\ge B(\psi)$, completing the proof.\end{proof}

We now turn to the upper estimate.

\begin{theorem}\label{upperest} If $M_\psi$ is bounded on \rm{$\LipW$} (or equivalently, \rm{$\LipWo$}), then
\begin{equation*}\|M_\psi\|_e\le A(\psi)+B(\psi).\end{equation*}
\end{theorem}

\begin{proof} Fix $n\in\N$, define the operator $K_n$ on $\LipW$ by
\begin{equation*}K_nf(v)=\begin{cases} f(v) &\hbox{ if }\quad |v|\le n,\\
f(v_n) &\hbox{ if }\quad |v|> n,\end{cases}\end{equation*}
where $f\in\LipW$ and $v_n$ is the ancestor of $v$ of length $n$. In particular, $K_nf\in\LipWo$ and
\ben K_nf(o)=f(o).\label{zerocondition}\eeqn
 Furthermore, $K_nf$ attains finitely many values, whose number does not exceed the number 
 of vertices in the closed ball centered at $o$ of radius $n$. Observe that if $\{g_k\}$ is a sequence in $\LipW$ with $\lipwnorm{g_k}\le 1$ for each $k\in\N$, then, $a=\displaystyle\sup_{k\in\N}|g_k(o)|\le 1$ so that $|K_ng_k(o)|\le a$. Furthermore, as a consequence of Proposition~\ref{modulus_est}, for each $v\in T^*$, and for each $k\in\N$, we have
$|K_ng_k(v)|\le 1+\log n$. Therefore, some subsequence $\{K_ng_{k_j}\}_{j\in\N}$ must converge to a function $g$ on $T$ attaining constant values on the sectors determined by the vertices on the sphere centered at $o$ of radius $n$. In particular, $g\in\LipW$ and since $K_ng_{k_j}\to g$ uniformly on the closed ball centered at $o$ of radius $n$, and $DK_ng_{k_j}$ and $Dg$ are 0 outside of the ball, we have
\ben \lipwnorm{K_ng_{k_j}-g}&=&|g_{k_j}(o)-g(o)|+\sup_{|v|\le n}|v|D(g_{k_j}-g)(v)\nonumber\\
&\le &|g_{k_j}(o)-g(o)|\nonumber\\
&\ &+n\sup_{|v|\le n}\left[|g_{k_j}(v)-g(v)|+\ |g_{k_j}(v^-)-g(v^-)|\right],\nonumber\eeqn
which converges to 0 as $j\to\infty$. Thus, $K_n$ is compact.

	Observe that the operator $M_\psi K_n$ is also compact. Furthermore, for $v\in T^*$, we have
\ben |v|D[(I-K_n)f](v)\le |v|Df(v) \le \lipwnorm{f}.\label{goodest}\eeqn
On the other hand, by Proposition~\ref{modulus_est}, we see that
\ben \left|[(I-K_n)f](v)\right|\le (1+\log |v|)\lipwnorm{f}.\label{best}\eeqn
We now use (\ref{best}) and (\ref{goodest}) to estimate $\lipwnorm{\psi(I-K_n)f}$:
\small \ben \lipwnorm{\psi(I-K_n)f}&= &\sup_{|v|>n}|v|\left|\psi(v)[(I-K_n)f](v)-\psi(v^-)[(I-K_n)f](v^-)\right|\nonumber\\
 &\le & \sup_{|v|>n}\left\{|\psi(v^-)||v|D[(I-K_n)f](v) + |v|D\psi(v)|[(I-K_n)f](v)|\right\}\nonumber\\
&\le &\sup_{|v|>n}|\psi(v^-)||v|D[(I-K_n)f](v)\nonumber\\
&\ &+\ \sup_{|v|>n}|v|\log|v|D\psi(v)\frac{|[(I-K_n)f](v)|}{1+\log|v|}\frac{1+\log|v|}{\log|v|}\nonumber\\&\le &\sup_{|v|>n}|\psi(v^-)|\lipwnorm{f}+\sup_{|v|>n}|v|\log|v|D\psi(v)\lipwnorm{f}\left(\frac{1+\log n}{\log n}\right).\nonumber\eeqn
\normalsize Therefore, using this estimate and taking the limit as $n\to\infty$, we obtain
\ben \|M_\psi\|_e&\le &\limsup_{n\to\infty}\|M_\psi-M_\psi K_n\|\nonumber\\
&=&\limsup_{n\to\infty}\sup_{\lipwnorm{f} = 1}\lipwnorm{\psi(I-K_n)f}\nonumber\\
&\le & A(\psi)+B(\psi),\nonumber\eeqn
completing the proof.\end{proof}

\section{Isometries and Isometric Zero Divisors}\label{isometries}
In this section we show that, in analogy to the case of the multiplication operators on the Lipschitz space of the tree \cite{ColonnaEasley:10}, there are no nontrivial isometric multiplication operators.

\begin{theorem}\label{iso_constant} The only isometric multiplication operators on \rm{$\LipW$} or \rm{$\LipWo$} are induced by the constant functions of modulus one.
\end{theorem}

\begin{proof} It is clear that the constant functions of modulus one are symbols of isometric multiplication operators on $\LipW$ and $\LipWo$. Thus, assume $M_\psi$ is an isometry on $\LipW$ or $\LipWo$ so that, in particular,
\ben \lipwnorm{\psi}=\lipwnorm{M_\psi 1}=1.\label{onenorm}\eeqn
 First we are going to show that $\psi$ has constant modulus 1. Fix $v\in T^*$ and let $f_v=\displaystyle\frac1{|v|+1}\c_v$. Then $\lipwnorm{f_v}=1$ and so $1=\lipwnorm{\psi f_v}=|\psi(v)|$. On the other hand, for $g=\frac12\chi_o$, $\lipwnorm{g}=2|g(o)|=1$ and thus \begin{equation*}1=\lipwnorm{\psi g}=2|\psi(o)g(o)|=|\psi(o)|.\end{equation*} Hence, $|\psi(v)|=1$ for all $v\in T$. From (\ref{onenorm}) it follows that $D\psi$ must be identically 0. Therefore, $\psi$ is a constant function of modulus one.\end{proof}

Inspired by \cite{ADMV}, we now define the notion of isometric zero divisor in a tree setting.

\begin{definition}\label{zero_divisor} Let $X$ be a Banach space of functions defined on a tree $T$ and let $Z$ be a nonempty subset of $T$. A function $\psi\in X$ is called a \emph{zero divisor} for $Z$ if it vanishes precisely at the vertices in $Z$ and $g/\psi\in X$ for every $g\in X$ vanishing on $Z$. The function $\psi$ is said to be an \emph{isometric zero divisor} if $\|g/\psi\|=\|g\|$ for each $g\in X$ which vanishes on $Z$.
\end{definition}

Recalling the set $\mathcal{P}$ in Proposition~\ref{dense}, we now see that, under certain hypotheses on the space $X$, the isometric zero divisors induce isometric multiplication operators on $X$.

\begin{theorem}\label{zero_divisor_iso} Let $X$ be a functional Banach space on $T$ containing $\mathcal{P}$ and satisfying the following properties:
\begin{enumerate}
\item[\normalfont{(a)}] $\mathcal{P}$ is dense in $X$.

\item[\normalfont{(b)}] For each $v\in T$ and each $f\in X$, $p_vf\in X$.
\end{enumerate}
\noindent If $\psi\in X$ is an isometric zero divisor, then $M_\psi$ is an isometry on $X$.\end{theorem}

\begin{proof} Let $\psi$ be an isometric zero divisor with zero set $Z$. Then, for each $p\in\mathcal{P}$, $p\psi\in X$ and vanishes at the vertices in $Z$, so
\ben \|p\|=\left\|\frac{p\psi}{\psi}\right\|=\|p \psi\|.\label{same_norm}\eeqn
We wish to show that $\psi f\in X$ and $\|\psi f\|=\|f\|$ for each $f\in X$. Fix $f\in X$ and using the density of $\mathcal{P}$, let $\{p_n\}$ be a sequence in $\mathcal{P}$ such that $\|p_n- f\|\to 0$ as $n\to\infty$. Since $\mathcal{P}$ is closed under addition, using (\ref{same_norm}), for $n,m\in\N$ we have
$\|p_n \psi-p_m\psi\|=\|p_n-p_m\|$, so $\{p_n\psi\}$ is a Cauchy sequence in $X$. By the completeness of $X$, there exists $g\in X$ such that $\|p_n \psi-g\|\to 0$ as $n\to\infty$. Since $X$ is a functional Banach space, the point evaluation functionals are bounded, hence $p_n\psi\to g$ and $p_n\to f$ pointwise in $T$. Therefore $g=\psi f$, proving that $\psi f\in X$. Moreover, by the triangle inequality and (\ref{same_norm}), we have
\ben |\|\psi f\|-\|f\||&\le& |\|\psi f\|-\|p_n \psi\||+|\|p_n\psi\|-\|p_n\||+|\|p_n\|-\|f\||\nonumber\\&\le& \|\psi f-p_n \psi\|+\|p_n-f\|\to 0,\nonumber\eeqn as $n\to\infty$. Therefore, $\|\psi f\|=\|f\|$, as desired.\end{proof}

We now turn our attention to the existence of isometric zero divisors on the spaces $\LipW$ and $\LipWo$.

\begin{corollary}\label{no_iso_zero} The space \rm{$\LipWo$} has no isometric zero divisors.
\end{corollary}

\begin{proof} By Theorem~\ref{iso_constant}, the only isometric multiplication operators are the constants of modulus one, which do not vanish anywhere. Observe that for each $v\in T$, recalling that $p_v$ is the characteristic function of the set consisting of $v$ and all its descendants, the set $\mathcal{P}$ is closed under multiplication by $p_v$. Thus, by Proposition~\ref{dense}, for each $v\in T$ and each $f\in\LipWo$, the function $p_v f\in \LipWo$. Therefore, since $\LipWo$ is a functional Banach space, the space $\LipWo$ satisfies the hypotheses of Theorem~\ref{zero_divisor_iso} and thus, no isometric zero divisors can exist.\end{proof}

\begin{theorem}\label{no_iso_zero_either} The space \rm{$\LipW$} has no isometric zero-divisors.
\end{theorem}

\begin{proof} By Corollary~\ref{no_iso_zero}, it suffices to show that the isometric zero divisors of $\LipW$ are in $\LipWo$.

Assume $\psi$ is an isometric zero divisor of $\LipW$. We begin by showing that $\psi$ is bounded. Fix $w\in T^*$ and define $f_w=\displaystyle\frac1{|w|}p_w$.
 Then, $f_w\in \mathcal{P}$ and for $v\in T^*$, we have
\begin{equation*}|v||f_w(v)-f_w(v^-)|=\frac{|v|}{|w|}\chi_w(v)=\chi_w(v),\end{equation*}
so $\lipwnorm{f_w}=1$. Therefore, $\lipwnorm{\psi f_w}=1$. Letting $D_w$ be the set of descendants of $w$, we have
\begin{equation*}\lipwnorm{\psi f_w}=\max\left\{|\psi(w)|,\sup_{v\in D_w}\frac{|v|}{|w|}|\psi(v)-\psi(v^-)|\right\}\ge |\psi(w)|.\end{equation*}
Hence, $|\psi(w)|\le 1$, proving the boundedness of $\psi$.

For $|w|\ge 2$, let us define \begin{equation*}g_w(v)=\begin{cases} 0&\quad\hbox{ if }v=o,\\
\frac1{|w|}\log|v| &\quad\hbox{ if }1\le |v|<|w|,\\
\frac1{|w|}\log|w| &\quad\hbox{ if }|v|\ge |w|.\end{cases}\end{equation*}
Since $g_w$ has finite support and by (\ref{chi_v}), the function $\chi_v$ is in $\mathcal{P}$, we deduce that $g_w\in\mathcal{P}$, so by (\ref{same_norm}), we obtain
\ben \ \ \ \ \ \ \lipwnorm{\psi g_w}=\lipwnorm{g_w}=\sup_{2\le |v|\le |w|}\frac{|v|}{|w|}\left(\log|v|-\log(|v|-1)\right)\le \frac{2\log 2}{|w|}.\label{norm_gw_estimate}\eeqn
On the other hand,
\ben \lipwnorm{\psi g_w}&\ge &\sup_{2\le |v|\le |w|}\frac{|v|}{|w|}\left|(\psi(v)-\psi(v^-))\log|v|+\psi(v^-)\log\frac{|v|}{|v|-1}\right|\nonumber\\
&\ge& \sup_{2\le |v|\le |w|}\frac{|v|}{|w|}|\psi(v)-\psi(v^-)|\log|v|-\sup_{2\le |v|\le |w|}\frac{|v|}{|w|}|\psi(v^-)|\log\frac{|v|}{|v|-1}\nonumber\\
&\ge& \sup_{2\le |v|\le |w|}\frac{|v|}{|w|}|\psi(v)-\psi(v^-)|\log|v|-\frac{2\log2}{|w|}\|\psi\|_\infty.\nonumber\eeqn
Therefore, using (\ref{norm_gw_estimate}), we get \begin{equation*}\sup_{2\le |v|\le |w|}\frac{|v|}{|w|}|\psi(v)-\psi(v^-)|\log|v|\le \frac{2\log 2(1+\|\psi\|_\infty)}{|w|}.\end{equation*}
Multiplying both sides by $|w|$ and letting $|w|\to \infty$, we obtain
\begin{equation*}\sup_{|v|\ge 2}|v||\psi(v)-\psi(v^-)|\log|v|<\infty.\end{equation*}
Hence $\lim\limits_{|v|\to\infty}|v||\psi(v)-\psi(v^-)|=0$, proving that $\psi\in\LipWo$.\end{proof}

\section*{Acknowledgements} 
The research of the first author is supported by a grant from the College of Science and Health of the University of Wisconsin-La Crosse.

\bibliographystyle{amsplain}
\bibliography{references.bib}

\providecommand{\bysame}{\leavevmode\hbox to3em{\hrulefill}\thinspace}
\providecommand{\MR}{\relax\ifhmode\unskip\space\fi MR }
\providecommand{\MRhref}[2]{%
  \href{http://www.ams.org/mathscinet-getitem?mr=#1}{#2}
}
\providecommand{\href}[2]{#2}
\begin{thebibliography}{10}

\bibitem{ADMV}
Alexandru Aleman, Peter Duren, Maria~J. Martin, and Dragan Vukoti\'{c},
  \emph{Multiplicative isometries and isometric zero-divisors}, Canad. J. Math.
  \textbf{62} (2010), no.~5, 961--974. \MR{2730350}

\bibitem{Cartier:72}
Pierre Cartier, \emph{Fonctions harmoniques sur un arbre}, Symposia
  {M}athematica, {V}ol. {IX} ({C}onvegno di {C}alcolo delle {P}robabilit\`a,
  {INDAM}, {R}ome, 1971), 1972, pp.~203--270. \MR{0353467}

\bibitem{CohenColonna:92}
Joel~M. Cohen and Flavia Colonna, \emph{The {B}loch space of a homogeneous
  tree}, vol.~37, 1992, Papers in honor of Jos\'{e} Adem (Spanish), pp.~63--82.
  \MR{1317563}

\bibitem{CohenColonna:94}
\bysame, \emph{Embeddings of trees in the hyperbolic disk}, Complex Variables
  Theory Appl. \textbf{24} (1994), no.~3-4, 311--335. \MR{1270321}

\bibitem{Colonna:89}
Flavia Colonna, \emph{Bloch and normal functions and their relation}, Rend.
  Circ. Mat. Palermo (2) \textbf{38} (1989), no.~2, 161--180. \MR{1029707}

\bibitem{ColonnaEasley:10}
Flavia Colonna and Glenn~R. Easley, \emph{Multiplication operators on the
  {L}ipschitz space of a tree}, Integral Equations Operator Theory \textbf{68}
  (2010), no.~3, 391--411. \MR{2735443}

\bibitem{Conway:07}
John~B. Conway, \emph{A course in functional analysis}, second ed., Graduate
  Texts in Mathematics, vol.~96, Springer-Verlag, New York, 1990. \MR{1070713}

\bibitem{DiBiasePicardello:95}
Fausto Di~Biase and Massimo~A. Picardello, \emph{The {G}reen formula and
  {$H^p$} spaces on trees}, Math. Z. \textbf{218} (1995), no.~2, 253--272.
  \MR{1318159}

\bibitem{DurenRombergShields:69}
Peter~L. Duren, Bernhard~W. Romberg, and Allen~L. Shields, \emph{Linear
  functionals on {$H^{p}$} spaces with {$0<p<1$}}, J. Reine Angew. Math.
  \textbf{238} (1969), 32--60. \MR{259579}

\bibitem{KoranyiPicardelloTaibleson:88}
Adam Kor\'{a}nyi, Massimo~A. Picardello, and Mitchell~H. Taibleson, \emph{Hardy
  spaces on nonhomogeneous trees}, Symposia {M}athematica, {V}ol. {XXIX}
  ({C}ortona, 1984), Sympos. Math., XXIX, Academic Press, New York, 1987, With
  an appendix by Picardello and Wolfgang Woess, pp.~205--265. \MR{951187}

\bibitem{Pavone:92}
Marco Pavone, \emph{Chaotic composition operators on trees}, Houston J. Math.
  \textbf{18} (1992), no.~1, 47--56. \MR{1159439}

\bibitem{Pavone:92-II}
\bysame, \emph{Toeplitz operators on discrete groups}, J. Operator Theory
  \textbf{27} (1992), no.~2, 359--384. \MR{1249652}

\bibitem{Pavone:93}
\bysame, \emph{Partially ordered groups, almost invariant sets, and {T}oeplitz
  operators}, J. Funct. Anal. \textbf{113} (1993), no.~1, 1--18. \MR{1214895}

\bibitem{RabinovichRoch:07}
Vladimir~S. Rabinovich and Steffen Roch, \emph{Essential spectra of difference
  operators on {$\Bbb Z^n$}-periodic graphs}, J. Phys. A \textbf{40} (2007),
  no.~33, 10109--10128. \MR{2371282}

\bibitem{RabinovichRoch:08}
\bysame, \emph{Fredholm properties of band-dominated operators on periodic
  discrete structures}, Complex Anal. Oper. Theory \textbf{2} (2008), no.~4,
  637--668. \MR{2465542}

\bibitem{RabinovichRoch:10}
\bysame, \emph{Finite sections of band-dominated operators on discrete groups},
  Recent progress in operator theory and its applications, Oper. Theory Adv.
  Appl., vol. 220, Birkh\"{a}user/Springer Basel AG, Basel, 2012, pp.~239--253.
  \MR{2953882}

\bibitem{Roe:05}
John Roe, \emph{Band-dominated {F}redholm operators on discrete groups},
  Integral Equations Operator Theory \textbf{51} (2005), no.~3, 411--416.
  \MR{2126819}

\bibitem{Ye:06}
Shanli Ye, \emph{Multipliers and cyclic vectors on the weighted {B}loch space},
  Math. J. Okayama Univ. \textbf{48} (2006), 135--143. \MR{2291174}

\bibitem{Zhu:07}
Kehe Zhu, \emph{Operator theory in function spaces}, second ed., Mathematical
  Surveys and Monographs, vol. 138, American Mathematical Society, Providence,
  RI, 2007. \MR{2311536}

\end{thebibliography}
\end{document}